\documentclass[11pt]{amsart}

\allowdisplaybreaks[4]

\newtheorem{theorem}{\bf Theorem}[section]
\newtheorem{corollary}{\bf Corollary}[section]
\newtheorem{lemma}{\bf Lemma}[section]
\newtheorem{proposition}{\bf Proposition}[section]

\theoremstyle{definition}
\newtheorem{definition}{\bf Definition}[section]

\theoremstyle{remark}
\newtheorem{remark}{\bf Remark}[section]

\numberwithin{equation}{section}

\def\C{\mathbb C}
\def\N{\mathbb N}
\def\Q{\mathbb Q}
\def\R{\mathbb R}

\def\Z{\mathbb Z}

\def\p{\mathfrak p}
\def\P{\mathfrak P}
\def\D{\mathfrak D}

\def\be{\begin{equation}}
\def\ee{\end{equation}}

\def\bea{\begin{eqnarray}}
\def\eea{\end{eqnarray}}

\def\beas{\begin{eqnarray*}}
\def\eeas{\end{eqnarray*}}

\begin{document}

\title[Uniform boundedness of Fourier partial sum operators]
{Uniform boundedness of the Fourier partial sum operators on the weighted spaces of local fields}

\author{Md. Nurul Molla}
\address{(Md. N. Molla) Statistics and Mathematics Unit, Indian Statistical Institute, 203 B. T. Road, Kolkata, 700108, India}
\email{bittu8bnc@gmail.com}	
	
\author{Biswaranjan Behera}
\address{(B. Behera) Statistics and Mathematics Unit, Indian Statistical Institute, 203 B.\,T. Road, Kolkata, 700108, India}
\email{biswa@isical.ac.in}


\subjclass{Primary: 43A70; Secondary: 42B25, 43A25}
\keywords{Weighted norm inequalities, Fourier series, maximal function, $A_p$ weights, totally disconnected groups, local fields, shift-invariant space, Schauder basis}

\begin{abstract}
Let $S_n f$ be the $n$th partial sum of the Fourier series of a function $f$ in $L^1(\D)$, where $\D$ is the ring of integers of a local field $K$. For $1<p<\infty$, we characterize all weight functions $w$ so that the partial sum operators $S_n$, $n\geq 0$, are uniformly bounded on the weighted space $L^p(\D, w)$ and that $S_n f$ converges to $f$ in $L^p(\D,w)$. This includes the case where $K$ is a $p$-adic number field or a field of formal Laurent series $\mathbb{F}_q((X))$ over a finite field $\mathbb{F}_q$, and in particular, when $\D$ is the Walsh-Paley or dyadic group $2^\omega$. As an application, in a local field $K$ of positive characteristic, we provide a necessary and sufficient condition on a function $\varphi\in L^2(K)$ for which the collection of translates of $\varphi$ forms a Schauder basis for its closed linear span. Moreover, we establish sharp bounds for the Hardy-Littlewood maximal operator. 
\end{abstract}
	
\maketitle

\section{Introduction and Main Results}
A field $K$ equipped with a topology is called a local field if both the additive and multiplicative groups of $K$ are locally compact abelian groups. The locally compact, nondiscrete, complete fields have been completely characterized and are either connected or are totally disconnected. If the field is connected, then it is either the field of real or complex numbers. The totally disconnected fields can be of characteristic zero or positive. If $K$ is of characteristic zero, then it is either  the $p$-adic field $\Q_p$ for some prime $p$, or a finite extension of such a field. If $K$ has positive characteristic, then it is a field of formal power series $\mathbb{F}_q((X))$ over a finite field $\mathbb{F}_q$, where $q=p^c$. If $c=1$, then it is a $p$-series field and if $c>1$, then $K$ is an algebraic extension of degree $c$ of a $p$-series field. For various aspects of harmonic analysis on local fields, we refer the reader to the monograph by Taibleson~\cite{Taib}. For the details of the classification of local fields, see Chapter~1 of~\cite{Taib} and Chapter~4 of~\cite{RV}.

Local fields have proved to be an ideal and effective setting for various problems in Euclidean harmonic analysis, wavelet analysis, etc. For example, for problems related to Carleson's theorem and time-frequency analysis, the study of Walsh-Fourier series on the ring of integers $\mathfrak{D}$ of the $2$-series field $\mathbb{F}_2((X))$ (which is the classical Walsh-Paley or dyadic group $2^\omega$) turned out to be very significant (see, e.g.,~\cite{DP,DL}). Recently, there have been considerably increased interest to study local field versions of several other problems in Euclidean harmonic analysis, geometric measure theory, nonlinear analysis, etc. In~\cite{Papa}, Papadimitropoulos proved the existence of Salem sets in the ring of integers of any local field and studied the Fourier restriction phenomenon on such sets. He also proved optimal extension of the Hausdorff-Young inequality in the local field setting (see~\cite{MR}). For further discussion on Fourier restriction and Kakeya problems over the local field $\Q_p$, see~\cite{HW}. Recently, Fan, Fan, Liao and Shi proved the Fugelede's conjecture on $\mathbb{Q}_p$ (see~\cite{FFLS} and~\cite{FFS}). Fan~\cite{Fan} and Shi~\cite{Shi} have extensively studied spectral measures and their variants on local fields. 

In this article, we study the weighted norm inequalities for the partial sum operators of Fourier series on the ring of integers of a general local field. Such inequalities were obtained for the circle group by Hunt, Muckenhoupt and Wheeden in~\cite{HMW} and for Vilenkin groups by Young in~\cite{You}. As we explain later, the techniques used in~\cite{HMW} and~\cite{You} are not adaptable to our setting. 

We now briefly present some notation which will allow us to state our results. For more details, we refer to Section~\ref{s.prelim}. Let $K$ be a local field. Then, there is an integer $q=p^c$, where $p$ is a prime and $c$ is a positive integer, and a natural norm $|\cdot|$ on $K$ such that $|x|=q^l$ for some integer $l$ if $x\in K$ and $x\neq 0$ and $|0|=0$. Let $dx$ be a Haar measure on the additive group $K^+$. Then, $d(\alpha x)=|\alpha|\,dx$ for all $\alpha\in K^*=K\setminus\{0\}$. Let $\D=\{x\in K:|x|\leq 1\}$ be the ring of integers in $K$. We normalize the Haar measure so that $|\D|=\int_\D 1\,dx=1$. Let $\Lambda=\{u(n):n\geq 0\}$ be a complete set of distinct coset representatives of $\D$ in $K$. Let $\chi$ be a fixed character of $K$ which is trivial on $\D$ but non-trivial on $\P^{-1}=\{x\in K :|x|\leq q\}$. For $y\in K$, define $\chi_y(x)=\chi(yx)$. Then, it follows that $\{\chi_{u(n)}: n\geq 0\}$ is a complete set of distinct characters on $\D$ so that it forms an orthonormal basis for $L^2(\D)$. A function $f$ on $K$ is said to be \emph{$\Lambda$-periodic} if $f\bigl(x+u(n)\bigr)=f(x)$ for all $n\geq 0$ and for a.e.\,$x\in K$.   

For $n \geq 0$, the $n$th Fourier coefficient of a function $f\in L^1(\D)$ is defined by 
\[
\hat{f}\bigl(u(n)\bigr)=\int_{\D}f(x)\overline{\chi_{u(n)}(x)}\,dx.
\]
The Fourier series of $f$ is 
\[
f(x)\sim\sum_{n=0}^\infty\hat{f}\bigl(u(n)\bigr)\chi_{u(n)}(x).
\]
For $n\geq 0$, let $S_n f$ be the $n$th partial sum of the Fourier series of $f$:
\[
S_n f(x)=\sum_{k=0}^{n-1}\hat{f}\bigl(u(k)\bigr)\chi_{u(k)}(x).
\]
A weight $w$ on $K$ is a locally integrable function on $K$ which is positive almost everywhere. The spaces $L^p(w)$ and $L^p(\D,w)$ are the usual spaces of $p$-integrable functions on $K$ and $\D$ respectively with respect to the measure $w(x)\,dx$. 

In~\cite{Taib1} and~\cite{Taib}, for $1<p<\infty$, Taibleson proved that the partial sum operators $S_n, n\geq 0$, are uniformly bounded on $L^p(\D)$ and that $S_n f$ converges to $f$ in $L^p(\D)$. In this article, we prove a weighted version of this result. We characterize all weights $w$ on $K$ such the operators $S_n$, $n\geq 0$, are uniformly bounded on the weighted spaces $L^p(\D,w)$, $1<p<\infty$, and that $S_n f$ converges to $f$ in $L^p(\D,w)$ for $1<p<\infty$. 

Our main result in this article is the following.

\begin{theorem}\label{main}
Let $K$ be a local field, $w$ be a $\Lambda$-periodic weight function on $K$ and $1<p<\infty$. Then, the following statements are equivalent.
\begin{enumerate}
\item[(a)] There is a positive constant $C$ such that
\begin{equation}\label{nec}
\Bigl(\frac{1}{|B|}\int_{B}w(x)\,dx\Bigr)\Bigl(\frac{1}{|B|}\int_{B}w(x)^{-\frac{1}{p-1}}\,dx\Bigr)^{p-1}\leq C
\end{equation}
for every ball $B$ in $K$.
\item[(b)] There is a positive constant $C_{p,w}$, depending only on $w$ and $p$, such that for every $f\in L^p(\D,w)$, we have $f\in L^1(\D)$ and 
\begin{equation}\label{weightS_nf}
\int_{\D}|S_{n}f(x)|^p w(x)\,dx\leq C_{p,w}\int_{\D}|f(x)|^pw(x)\,dx
\end{equation}
for $n=0,1,2,\dots$.
\item[(c)] For every $\Lambda$-periodic function $f\in L^p(\D,w)$, we have $f\in L^1(\D)$ and $S_nf\rightarrow f$ in $L^p(\D,w)$ as $n\rightarrow\infty$. 
\end{enumerate}
\end{theorem}

For $1<p<\infty$, a weight $w$ on $K$ satisfying~\eqref{nec} is called an $A_p$ weight. The infimum of all such $C$ appearing in~\eqref{nec} is called the $A_p$ characteristic of $w$ and is denoted by $[w]_{A_p}$.

The corresponding classical result on the Euclidean spaces was proved by Hunt, Muckenhoupt and Wheeden in~\cite{HMW}. For the case (a) implies (b), the weighted weak type inequalities for the conjugate function $\tilde{f}$ for $1<p<2$ plays an important role in their proof. This linear operator is given by convolution with the distribution p.v.~$\frac{1}{\tan\frac{\phi}{2}}$. A crucial step in their proof is the decomposition of bounded mean oscillation functions: If $1\leq p<\infty$, $w$ satisfies $A_p$ condition and has period $2\pi$, then $\log w$ is of bounded mean oscillation. Moreover, there exist bounded functions $u$, $v$ with period $2\pi$ such that $\log w(x)=u(x)+\tilde{v}(x)$ for almost every $x$ (see page $240$ of~\cite{HMW}). Such a decomposition relies on the theory of Hardy spaces on the unit disc. It is not clear, a priori, if the required decomposition holds in the context of local fields. Another advantage in Euclidean spaces is the simple and closed form of the Dirichlet kernel $D_n(x)=\frac{1}{2\pi}\frac{\sin(n+\frac{1}{2})x}{\sin\frac{x}{2}}$, $n\in\N$, which turned out to be very useful in the deduction of (b) implies (a). 
 
In~\cite{You}, Young proved a similar result for the Vilenkin groups. A Vilenkin group $G$ is a direct product of cyclic groups of order $p_i$, where each $p_i$ is an integer greater than or equal to 2. In particular, if we take each $p_i$ to be equal to a prime $p$, then $G$ becomes the ring of integers of the $p$-series field $\mathbb{F}_p((X))$. On the other hand, the ring of integers of a $p$-adic field $\Q_p$ is not a Vilenkin group. Our result in this article is on the ring of integers $\D$ of a general local field which, of course, includes $\Q_p$ and $\mathbb{F}_p((X))$ as special cases. The Walsh-Paley or dyadic group $2^\omega$ can be identified with the additive group of the ring of integers of the field $\mathbb{F}_2((X))$. Hence, our result is also valid for Walsh-Paley group. It is also applicable to the $p$-adic fields $\Q_p$ which are not included in Young's result. 

Young's proof consists of adapting the methods of~\cite{HMW} and~\cite{CF} to the Vilenkin-Fourier series. Similar to the circle group, the conjugate function again plays an important role in this case also. Moreover, the characters of a Vilenkin group can be written as a finite product of exponential functions and the Dirichlet kernels can be expressed in terms of sine functions as in the case of the trigonometric Fourier series. This fact was exploited in~\cite{You}. In a general local field, there are simple expressions for the Dirichlet kernel $D_n$ only for $n=q^k$, $k\geq 0$, but computing $D_n$ for general values of $n$ is not easy and there is no simple expression unlike the case of trigonometric and Vilenkin Fourier series. These are some of the obstacles which makes our problems somewhat different from that of the circle and Vilenkin groups. Nevertheless, some aspects of the proofs of the implications (b) implies (a) and (c) implies (b) can be adapted from those of~\cite{HMW}. In both instances we omit the details to avoid repetition. The main difficulty lies in the proof of (a) implies (b). Our approach relies on some techniques from the theory of regular singular integral operators (see, e.g.,~Chapter IV of~\cite{GR}). Hence, this also gives another proof for the case of Vilenkin groups which are products of cyclic groups of the same order $p$.

We provide an application of our main result to the theory of Schauder bases on shift-invariant spaces. Before stating this result, we consider a more general problem. Let $K$ be a local field, $\Gamma$ be a countable set in $K$, and $\varphi\in L^2(K)$. Define
\[
V(\varphi,\Gamma)=\overline{\rm span}\{\varphi(\cdot-\gamma): \gamma\in\Gamma\}.
\]
We consider the following problem: When does the system of translates  $\{\varphi(\cdot-\gamma):\gamma\in\Gamma\}$ form an orthonormal basis/Scahuder basis for the space $V(\varphi,\Gamma)$?

If $K=\Q_p$, the field of $p$-adic numbers, then we can answer this question using Fuglede's conjecture on $\Q_p$.

\begin{theorem}\label{fuglede}
A Borel set $\Omega$ of positive and finite Haar measure in $\Q_p$ is a spectral set if and only if it tiles $\Q_p$ by translations.
\end{theorem}

We say that $\Omega$ tiles $\Q_p$ by translations if there exists a set $T\subset\Q_p$ such that $\sum_{t\in T}{\bf 1}_\Omega(x-t)=1$ for a.e.\,$x\in\Q_p$, where ${\bf 1}_\Omega$ is the characteristic function of $\Omega$. The set $\Omega$ is said to be a spectral set if there exists a set $S\subset\Q_p$ such that $\{\chi_s:s\in S\}$ is an orthonormal basis for $L^2(\Omega)$. Theorem~\ref{fuglede} was recently proved by Fan, Fan, Liao, and Shi~\cite{FFLS}. Using this theorem we can prove the following result. 

\begin{theorem}\label{onb}
Let $\Omega $ be a Borel set of positive and finite Haar measure in $\Q_p$. Let $\varphi \in L^2(\Q_p)$ with ${\rm supp}\,\hat\varphi\subseteq\Omega$ and $|\hat{\varphi}|=1$ on $\Omega$ a.e. If $\Omega$ tiles $\Q_p$ by translations, then there exists $\Gamma\subset\Q_p$ such that $\{\varphi(\cdot-\gamma):\gamma\in\Gamma\}$ forms an orthonormal basis for $V(\varphi,\Gamma)$. The converse is also true.
\end{theorem}

Fuglede's conjecture is still open for general local fields, in particular for local fields of positive characteristic. Nevertheless, for the set $\Gamma=\{u(k):k\in\N_0\}$, we characterize the functions $\varphi$ in terms of $A_2$ weights. This extends the result from the real line obtained in~\cite{NS}.

\begin{theorem}\label{sc}
Let $K$ be a local field of positive characteristic, $\Gamma=\{u(k):k\in\N_0\}$, and $\varphi\in L^2(K)$. Then, the family $\{\varphi(\cdot-u(k)): k\in\N_0\}$ is a Schauder basis for  $V(\varphi, \Gamma)$ if and only if $w_\varphi\in A_2(\D)$, where $w_{\varphi}(\xi)=\sum_{n=0}^\infty|\hat{\varphi}(\xi+u(n))|^2$.
\end{theorem}

The rest of the article is organized as follows. In Section~\ref{s.prelim}, we fix terminologies and recall some basic facts about local fields. In Section~\ref{s.HLM}, we discuss about the boundedness of the Hardy-Littlewood maximal operator on weighted spaces on local fields. We also talk about the best possible power dependence of the operator norm on the $A_p$ characteristic. We prove our main result Theorem~\ref{main} in Section~\ref{s.proof}. In Section~\ref{s.appl}, using this result, we prove Theorem~\ref{sc} and conclude our article with a proof of Theorem~\ref{onb}.

\section{Preliminaries}\label{s.prelim}

Let $K$ be a field and a topological space. Then $K$ is called a locally compact field or a \emph{local field} if both $K^+$ and $K^*$ are locally compact abelian groups, where $K^+$ and $K^*$ denote the additive and multiplicative groups of $K$ respectively.

If $K$ is any field and is endowed with the discrete topology, then $K$ is a local field. Further, if $K$ is connected, then $K$ is either $\R$ or $\C$. If $K$ is not connected, then it is totally disconnected. So, by a local field, we mean a field $K$ which is locally compact, non-discrete and totally disconnected. We use the notation of the book by Taibleson~\cite{Taib}. Proofs of all the results stated in this section can be found in the books~\cite{Taib} and~\cite{RV}.

Let $K$ be a local field. Since $K^+$ is a locally compact abelian group, we choose a Haar measure $dx$ for $K^+$. If $\alpha\in K$ and $\alpha\neq 0$, then $d(\alpha x)$ is also a Haar measure. By uniqueness of Haar measure, $d(\alpha x)=c\,dx$ for some $c>0$. Let $c=|\alpha|$. We call $|\alpha|$ the \emph{absolute value} or the  \emph{valuation} of $\alpha$. We also let $|0| = 0$.

The map $x\rightarrow |x|$ has the following properties:
\begin{itemize}
\item[(a)] $|x| = 0$ if and only if $x=0$;
\item[(b)] $|xy|=|x||y|$ for all $x,y\in K$;
\item[(c)] $|x+y|\leq\max\{|x|, |y|\}$ for all $x,y\in K$.
\end{itemize}
Property (c) is called the \emph{ultrametric inequality}. It follows that 
\begin{equation}\label{e.max}
|x+y|=\max\{|x|, |y|\}\quad\mbox{if}~|x|\ne |y|.
\end{equation}
The set $\D=\{x\in K:|x|\leq 1\}$ is called the \emph{ring of integers} in $K$. It is the unique maximal compact subring of $K$. Define $\P=\{x\in K:|x|<1\}$. The set $\P$ is called the \emph{prime ideal} in $K$. The prime ideal in $K$ is the unique maximal ideal in $\D$. It is principal and prime. 

Since $K$ is totally disconnected, the set of values $|x|$, as $x$ varies over $K$, is a discrete set of the form $\{s^k: k\in\Z\}\cup \{0\}$ for some $s>0$. Hence, there is an element of $\P$ of maximal absolute value. Let $\p$ be a fixed element of maximum absolute value in $\P$. Such an element is called a \emph{prime element} of $K$. It can be proved that $\D$ is compact and open. Hence, $\P$ is compact and open. Therefore, the residue space $\D/\P$ is isomorphic to a finite field $\mathbb{F}_q$, where $q=p^c$ for some prime $p$ and $c\in\N$. For a proof of this fact, we refer to~\cite{Taib}.

For a measurable subset $E$ of $K$, let $|E|=\int_K{\mathbf 1}_E(x)\,dx$, where ${\mathbf 1}_E$ is the characteristic function of $E$ and $dx$ is the Haar measure of $K$ normalized so that $|\D|=1$. Then, it is easy to see that $|\P|=q^{-1}$ and $|\p|=q^{-1}$ (see~\cite{Taib}). It follows that if $x\neq 0$, and $x\in K$, then $|x|=q^k$ for some $k\in\mathbb{Z}$.

Let $\D^*=\D\setminus\P=\{x\in K: |x|=1\}$. It is the group of units in $K^*$. If $x\neq 0$, we can write $x=\p^k x'$, with $x'\in\D^*$. Let $\P^k=\p^k\D=\{x\in K:|x|\leq q^{-k}\}$, $k\in\Z$. These are called \emph{fractional ideals}. Each $\P^k$ is compact and open and is a subgroup of $K^+$ (see~\cite{RV}). It follows that $|\P^k|=q^{-k}$ for $k\in\Z$.

A set $B$ of the form $h+\P^k$ will be called a \emph{ball with centre} $h$ \emph{and radius} $q^{-k}$. 
It is easy to verify the following facts. For a proof, we refer to~\cite{Taib}.

\begin{proposition}\label{p.balls}
In a local field $K$
\begin{enumerate}
\item[(a)] every point of a ball is its centre,
\item[(b)] if two balls intersect, then one contains the other.
\end{enumerate}
\end{proposition}

For $k\in\Z$, let $\Phi_k={\bf 1}_{\P^k}$, the characteristic function of $\P^k$\index{$\Phi_k$}. Note that the characteristic function of the ball $h+\P^k$ is $\tau_h\Phi_k=\Phi_k(\cdot-h)$. It follows from Proposition~\ref{p.balls}\,(b) that $\tau_h\Phi_k$ is constant on cosets of $\P^k$.

The space of all finite linear combinations of functions of the form $\tau_h\Phi_k$, $h\in K$, $k\in\Z$, will be denoted by $\mathcal{S}$. The following result provides information about the structure of the set $\mathcal{S}$.

\begin{proposition}\label{supp}
If $g\in\mathcal{S}$ is constant on cosets of $\P^k$ and is supported on $\P^l$, then $\hat g\in\mathcal{S}$ is constant on cosets of $\P^{-l}$ and is supported on $\P^{-k}$.
\end{proposition}

\begin{proof}
Suppose $g$ is constant on cosets of $\P^k$ and is supported on $\P^l$. Take $h\in\P^{-l}$ so that $|h|\leq q^l$. Now,
\[
\hat g(u+h)=\int_K g(x)\overline{\chi_{u+h}(x)}\,dx
=\int_{|x|\leq q^{-l}} g(x)\overline{\chi(ux)}\overline{\chi(hx)}\,dx
=\hat g(u),
\]
since $|xh|\leq 1$ and $\chi$ is trivial in $\D$. Hence, $\hat g$ is constant on cosets of $\P^{-l}$. Also, since $g$ is constant on cosets of $\P^k$, we have
\[
\hat g(u)=\int_K g(x)\overline{\chi(ux)}\,dx=\int_K g(x+h)\overline{\chi(ux)}\,dx
\]
for all $h\in\P^k$. By changing variables, we get
\[
\hat g(u)=\int_K g(x)\overline{\chi(ux-uh)}\,dx=\chi(uh)\int_K g(x)\overline{\chi(ux)}\,dx=\chi(uh)\hat g(u).
\]
If $u\not\in{\P}^{-k}$, then $|u|>q^k$. Hence, there exists $h\in\P^k$ such that $|uh|>1$ and $\chi(uh)\not=1$. Therefore, $\hat g(u)=0$.
\end{proof}

Let $\chi_u$ be any character of $K^+$. Since $\mathfrak{D}$ is a subgroup of $K^+$, the restriction $\chi_{u}|_\mathfrak{D}$ is a character of $\mathfrak{D}$. Also, as characters of $\mathfrak{D}, \chi_u = \chi_v$ if and only if $u-v\in \mathfrak{D}$. That is, $\chi_u=\chi_v$ if $u+\mathfrak{D}=v+\mathfrak{D}$ and $\chi_u\neq \chi_v$ if $(u+\mathfrak{D})\cap (v+\mathfrak{D})=\emptyset$. Hence, if $\{u(n)\}_{n=0}^{\infty}$ is a complete list of distinct coset representative of $\mathfrak{D}$ in $K^+$, then $\{\chi_{u(n)}\}_{n=0}^{\infty}$ is a list of distinct characters of $\mathfrak{D}$. It is proved in~\cite{Taib} that this list is complete. 

\begin{proposition}\label{p.com}
Let $\{u(n)\}_{n=0}^{\infty}$ be a complete list of distinct coset representatives of $\mathfrak{D}$ in $K^+$. Then $\{\chi_{u(n)}\}_{n=0}^{\infty}$ is a complete list of distinct characters of $\D$. Moreover, it is an orthonormal basis for $L^2(\D)$.
\end{proposition}

Given such a list of characters $\{\chi_{u(n)}\}_{n=0}^{\infty}$, we define the \emph{Fourier coefficients} of $f\in L^1(\D)$ as
\[
\hat{f}\bigl(u(n)\bigr)=\int_{\D}f(x)\overline{\chi_{u(n)}(x)}dx.
\]
For $n\geq 0$, let $S_n f$ be the $n$th partial sum of the Fourier series of $f$:
\[
S_n f(x)=\sum_{k=0}^{n-1}\hat{f}\bigl(u(k)\bigr)\chi_{u(k)}(x).
\]
The series $\sum_{n=0}^{\infty}\hat{f}(u(n))\chi_{u(n)}(x)$ is called the \emph{Fourier series} of $f$. 

Let $\N_0=\N\cup \{0\}$. We will now write $\chi_n=\chi_{u(n)}|_\D$ for $n\in\N_0$. The \emph{Dirichlet kernels} are the functions
\[
D_n(x)=\sum_{k=0}^{n-1}\chi_k(x), n\geq 1,~\mbox{and}~D_0(x)\equiv 0. 
\]
From the standard $L^2$-theory for compact abelian groups, we conclude that the Fourier series of $f$ converges to $f$ in $L^2(\D)$ and \emph{Parseval's identity} holds:
\[
\int_{\D}|f(x)|^2\,dx=\sum\limits_{n\in\N_0}|\hat{f}\bigl(u(n)\bigr)|^2.
\]
Also, if $f\in L^1(\mathfrak{D})$ and $\hat f\bigl(u(n)\bigr)=0$ for all $n\in\N_0$, then $f=0$ a.e.

These results hold irrespective of the ordering of the characters. We now proceed to impose a natural order on the sequence $\{u(n):n\in\N_0\}$. Note that $\Gamma=\mathfrak{D}/\mathfrak{P}$ is isomorphic to the finite field $\mathbb{F}_q$ and $\mathbb{F}_q$ is a $c$-dimensional vector space over the field $\mathbb{F}_p$. We choose a set $\{1=\epsilon_0, \epsilon_1, \epsilon_2, \dots, \epsilon_{c-1}\}\subset\mathfrak{D}^*$ such that span$\{\epsilon_j\}_{j=0}^{c-1}\cong\mathbb{F}_q$.
For $n\in \N_0$ such that $0\leq n< q$, we have
\[
n=a_0+a_1 p+\dots+a_{c-1} p^{c-1},\quad 0\leq a_k<p, \quad k=0,1,\dots,c-1.
\]
Define
\begin{equation}\label{e.undef1}
u(n)=(a_0+a_1\epsilon_1+\dots+a_{c-1}\epsilon_{c-1})\mathfrak{p}^{-1}.
\end{equation}
Note that $\{u(n):n=0, 1,\dots, q-1\}$ is a complete set of coset representatives of $\D$ in $\P^{-1}$ so that we can write
\begin{equation}\label{e.pinvd}
\P^{-1}=\bigcup_{l=0}^{q-1}\bigl(u(l)+\D\bigr). 
\end{equation}

Now, for $n\geq 0$, write
\[
n=b_0+b_1q+b_2q^2+\dots+b_sq^s,\quad 0\leq b_k<q,\quad k=0,1,2,\dots,s,
\]
and define
\begin{equation}\label{e.undef2}
u(n)=u(b_0)+u(b_1)\mathfrak{p}^{-1}+\dots+u(b_s)\mathfrak{p}^{-s}.
\end{equation}
This defines $u(n)$ for all $n\in\N_0$. In general, it is not true that $u(m+n)=u(m)+u(n)$. But it follows that
\begin{equation}\label{eq.un}
u(rq^k+s)=u(r)\mathfrak{p}^{-k}+u(s)\quad{\rm if}~r\geq 0, k\geq 0~{\rm and}~0\leq s <q^k.
\end{equation}

In the following proposition we list some properties of $\{u(n): n\in\N_0\}$. We refer to~\cite{BJ} for a proof. 

\begin{proposition}\label{p.un}
For $n\in\N_0$, let $u(n)$ be defined as in \eqref{e.undef1} and \eqref{e.undef2}. Then
\begin{enumerate}
\item[(a)] $u(n)=0$ if and only if $n=0$. If $k\geq 1$, then $|u(n)|=q^k$ if and only if $q^{k-1}\leq n<q^k$.
\end{enumerate}
Moreover, if $K$ is a local field of positive characteristic, then
\begin{enumerate}
\item[(b)] $\{u(k):k\in\N_0\}=\{-u(k):k\in\N_0\}$;
\item[(c)] for a fixed $l\in\N_0$, we have $\{u(l)+u(k): k\in\N_0\}=\{u(k):k\in\N_0\}$.
\end{enumerate}
\end{proposition}

\subsection{Weight functions and their properties}
As mentioned earlier, a weight $w$ on $K$ is a locally integrable function which is positive a.e. For any measurable subset $A$ of $K$, we define $w(A)=\int_A w(x)\,dx$. We say that a weight $w$ has the \emph{doubling property} if there exists a constant $C>0$ such that
\[
w(x+\P^{k-1})\leq C w(x+\P^k)
\]
for all $k\in\Z$ and $x\in K$. We mentioned about the notion of $A_p$ weights in the introduction. A weight $w$ is said to be in the \emph{Muckenhoupt $A_p$ class} if there exists a constant $C>0$ such that
\[
\Bigl(\frac{1}{|B|}\int_{B}w(x)\,dx\Bigr)\Bigl(\frac{1}{|B|}\int_{B}w(x)^{-\frac{1}{p-1}}\,dx\Bigr)^{p-1}\leq C
\]
for every ball $B$ in $K$. In this case, we also say that $w$ is an $A_p$ weight or $w\in A_p$. The infimum of all such $C$ appearing above is called the \emph{$A_p$ characteristic of $w$} and is denoted by $[w]_{A_p}$. Now, we mention some properties of the $A_p$ weights which will be used later. The proofs are similar to the corresponding results on the Euclidean spaces.

\begin{proposition} 
Let $1<p<q<\infty$ and $p'$ denote the conjugate index of $p$. Then
\begin{enumerate} 
\item[(a)] $A_p\subset A_q$,
\item[(b)] $w\in A_p$ if and only if $w^{1-p'}\in A_{p'}$.
\end{enumerate}
\end{proposition}

We will also need the following result. For a proof, we refer to~\cite{CH}. 

\begin{theorem}\label{rhi}
Let $1<p<\infty$ and $w\in A_p$. Then there exists a constant $C$ and $\epsilon>0$, depending only on $p$ and $[w]_{A_p}$, such that for any ball $B$
\[
\Bigl(\frac{1}{|B|}\int_B w(y)^{1+\epsilon}\,dy\Bigr)^{\frac{1}{1+\epsilon}}\leq C\frac{1}{|B|}\int_B w(y)\,dy.
\]
\end{theorem}

The above inequality is known as the \emph{reverse H\"older inequality}. The following simple corollaries of the reverse H\"older inequality will be useful later.

\begin{corollary}\label{cor}
Let $w\in A_p$, $1<p<\infty$. Then, there exists $r$ with $1<r<p$ such that $w$ is also in $A_r$.
\end{corollary}

\begin{corollary}\label{cor2}
Let $w\in A_p$, $1<p<\infty$. Then, there exists a $\delta>0$ and $C>0$ such that for any ball $B$ and for any subset $E$ of $B$, we have
\begin{equation}\label{Ainfty}
{w(E)}{|B|^\delta}\leq C{w(B)}{|E|^\delta}. 
\end{equation}
\end{corollary}

\subsection{Maximal function}
For $f\in L^1_{\rm loc}(K)$, the \emph{Hardy-Littlewood maximal function} $Mf$ is defined analogously as in the Euclidean spaces by
\[
Mf(x)=\sup_B\frac{1}{|B|}\int_B|f(y)|\,dy, \quad x\in K,
\]
where the supremum is taken over all balls $B$ containing $x$. 

The balls in a local field are dyadic in nature, in the sense that given any two balls, either they are disjoint or one ball contains the other (see Proposition~\ref{p.balls}). So, the maximal function $M$ is essentially like the dyadic maximal function in $\R^n$. If a function $f$ on $\R$ is supported on the positive real axis, then the dyadic maximal function of $f$ is zero almost everywhere on the negative axis. On the other hand, in a local field $K$, if $Mf(x_0)=0$ at some $x_0\in K$ with $|x_0|=q^k$, then $\P^{-l}$ is a ball that contains $x_0$ whenever $l\geq k$. Hence, the average of $|f|$ is zero over all  such balls. This shows that $f=0$ a.e. Therefore, in contrast to the dyadic maximal function in $\R^n$, $Mf$ can never be zero at any point in a local field unless $f=0$ a.e.

Next, for $1<s<\infty$, we define $M_s f(x)=\bigl(M|f|^s(x)\bigr)^\frac{1}{s}$. We also define the \emph{sharp maximal function} $f^\sharp$ analogously as in $\R^n$. For $f\in L^1_{\rm loc}(K)$,
\[
f^\sharp(x)=\sup_B\frac{1}{|B|}\int_{B}|f(y)-f_B|\,dy, 
\]
where $f_B=\frac{1}{|B|}\int_{B}f(x)\,dx$ is the average of $f$ over the ball $B$, and the supremum is taken over all balls $B$ containing $x$. 

We finish this section with a result that relates the $L^p(w)$-norms of $Mf$ and $f^\sharp$. This will be used in the proof of our main theorem. 
 
\begin{theorem}\label{T.sharp}
Let $1<p<\infty$ and $w\in A_p$. Then, there exists a constant $C_{p,w}>0$, depending only on $p$ and $w$, such that
\begin{equation}\label{Mtosharp}
\int_K {Mf(x)}^p w(x)\,dx\leq C_{p,w}\int_K {f^\sharp(x)}^p w(x)\,dx 
\end{equation} 
for every $f\in L^p(w)$. 
\end{theorem}

\begin{proof}
We refer the reader to Theorem $2.20$ of~\cite{GR} for a proof in the Euclidean spaces. A quick inspection shows that the arguments given there use Corollary~\ref{cor2}, the weak $(1,1)$ property of the maximal operator $M$ and Calder\'on-Zygmund decomposition. The local field version of Calder\'on-Zygmund  decomposition can be found in~\cite{Phillips} and~\cite{Taib}. Hence, the assertion carries over to our setting as well.
\end{proof}


\section{Maximal function and $A_p$ weights}\label{s.HLM}
On the Euclidean spaces, it is well known that the Hardy-Littlewood maximal operator is bounded on $L^p(w)$ for $A_p$ weights. Several proofs of this fact are available in the literature and almost all of these proofs can be adapted to the setting of local fields, for example, see~\cite{CH}. In~\cite{CH}, the authors assumed that the weights $w$ satisfy a doubling condition. For our proof of the main result, we need this result which is applicable to weights not necessarily having the doubling property. In this section, we improve their result by removing the doubling condition. We also talk about the relation between the operator norm of $M$ on $L^p(w)$ and the $A_p$ characteristic of $w$ in the sense of Buckley~\cite{Buck}. 

The theory of weighted norm inequalities for the maximal operator can also be extended to the more general setting of the spaces of homogeneous type, see~\cite{Hyto}. A space of homogeneous type is a quasi-metric space $X$ with quasi-metric $d$ such that the $d$-balls are open sets, and $\mu$ is a regular measure defined on the $\sigma$-algebra containing the $d$-balls that satisfies the ``doubling condition", i.\,e., there is a constant $A$ such that the measure of a ball of radius $2r$ is at most $A$ times the measure of the ball of radius $r$ with the same centre. Observe that local fields are spaces of homogeneous type. In another direction, Sauer~\cite{Sau} considered this problem on general locally compact abelian groups satisfying certain covering properties. In the same setting, Paternostro and Rela~\cite{PR} established the sharp bound for the norm of the maximal operator. Now we briefly describe their setting and the corresponding results and deduce the boundedness of the maximal operator $M$ on weighted spaces of local fields without any restriction on the weights. 

Let $G$ be a locally compact abelian group with a measure $\mu$ that is inner regular and such that $\mu(K)<\infty$ for every compact set $K\subset G$. Edwards and Gaudry~\cite{EG} defined the concept of covering families in a locally compact abelian group as follows.

\begin{definition}
Let $G$ be a locally compact abelian group. A collection $\{U_i:i\in\Z\}$ is a \emph{covering family} for $G$ if
$\{{U_i}:i\in\Z\}$ is an increasing base of relatively compact neighbourhoods of $0$, $\bigcup_{i\in\Z}U_{i}=G$, $\bigcap_{i\in \Z}U_{i}=\{0\}$, and there exists a constant $D\geq 1$ and an increasing function $\theta:\Z\rightarrow\Z$ such that for any $i\in \Z$ and any $x\in G$
\begin{itemize}
\item[(a)] $i\leq\theta(i)$;
\item[(b)] $U_{i}-U_{i}\subset U_{\theta(i)}$;
\item[(c)] $\mu(x+U_{\theta(i)})\leq D\mu(x+U_{i})$.  
\end{itemize}
\end{definition}

For each $x\in G$, the set $x+U_i$ will be called a \emph{base set} and the collection of all base sets will be denoted by $\mathcal{B}=\{x+U_i:x\in G, i\in\Z\}$. On such locally compact abelian groups, Edwards and Gaudry~\cite{EG} defined the analogue of Hardy-Littlewood maximal operator using these base sets and proved its boundedness on the $L^p$ spaces. In~\cite{Sau}, Sauer proved a weighted bound for the maximal operator on locally compact abelian groups and asked whether it is possible to obtain a sharp result as in Buckley~\cite{Buck} in this setting. Recently, Paternostro and Rela~\cite{PR} answered this question affirmatively. 

\begin{theorem}[\cite{PR}]\label{lca}
Let $M$ be the Hardy-Littlewood maximal function on $G$, $1<p<\infty$ and $w\in A_p$. Then there exists a constant $C_p>0$ such that 
$\|M\|_{L^p(w)\rightarrow L^p(w)}\leq C_p[w]_{A_p}^{\frac{1}{p-1}}$.
\end{theorem}

In the present case of local fields, we take $U_i=\P^i$, $i\in\Z$, $\theta$ to be the identity function on $\Z$ and $\mu$ is the Haar measure on $K$. Then, it is easy to see that $\{{\P^i}:i\in\Z\}$ is a covering family of $K$. Hence, Theorem~\ref{lca} can be applied to local fields to obtain the boundedness of the Hardy-Littlewood maximal operator. In~\cite{CH}, the authors assumed that the weights $w$ satisfy a doubling condition. By Theorem~\ref{lca}, we observe that this assumption is not needed.

In the Euclidean spaces $\R^n$, Buckley~\cite{Buck} proved that $\|M\|_{L^p(w)\rightarrow L^p(w)}\leq C_p [w]_{A_p}^{\frac{1}{p-1}}$ and that the power $[w]_{A_p}^{\frac{1}{p-1}}$ is best possible. We will show that an analogous result also holds in local fields. First, we construct weights on $K$ satisfying the doubling condition but not in $A_p$.

\subsection{Examples of doubling weights} 
Let $w(x)=|x|^\alpha$. We will show that $w$ is a doubling weight if $\alpha>-1$. We first compute the integral $\int_{\P^k}|x|^\alpha\,dx$. We have
\begin{eqnarray}
w(\P^k)=\int_{\P^k}|x|^\alpha \,dx 
& = & \sum_{l=-\infty}^{-k}\int_{\{x:|x|=q^l\}}|x|^\alpha \,dx \nonumber\\
& = & \sum_{l=-\infty}^{-k}q^{\alpha l}|\{x:|x|=q^l\}| \nonumber \\
& = & \sum_{l=-\infty}^{-k}q^{\alpha l}(q^l-q^{l-1}) \nonumber \\
& = & \frac{(q-1)q^{\alpha -k(\alpha +1)}}{q^{\alpha +1}-1}. \label{e.wofpk}
\end{eqnarray}
Note that $\alpha>-1$ is necessary for the series in the above computation to converge which guarantees the integrability of $w$. Now, let $B$ be any ball. Then $B$ is of the form $x_0+\P^k$ for some $x_0\in K$ and $k\in\Z$. We consider the cases (i)~$x_0\in\P^{k-1}$ and (ii)~$x_0\not\in\P^{k-1}$. 

(i)~If $x_0\in\P^{k-1}$, then $x_0+\P^{k-1}=\P^{k-1}$. Therefore, by~\eqref{e.wofpk}, we have
\[
w(x_0+\P^{k-1})=w(\P^{k-1})=\frac{(q-1)q^{\alpha-k(\alpha+1)}}{q^{\alpha+1}-1}q^{\alpha +1}.
\]
Now, we compute $w(x_0+\P^{k})$. If $x_0$ is also in $\P^k$, then $x_0+\P^{k}=\P^k$ so that
\begin{equation}\label{eq}
w(x_0+\P^k)=w(\P^k)=\frac{(q-1)q^{\alpha-k(\alpha+1)}}{q^{\alpha+1}-1},
\end{equation}
and we get $w(x_0+\P^{k-1})=q^{\alpha+1}w(x_0+\P^k)$. If $x_0\in\P^{k-1}\setminus\P^k$, then for any $x\in x_0+\P^k$, $|x|=|x_0|$. Hence, $\int_{x_0+\P^k}|x|^\alpha\,dx=q^{-k-k\alpha+\alpha}$. Thus, we get $w(x_0+\P^{k-1})=\frac{(q-1)q^{\alpha+1}}{q^{\alpha+1}-1}w(x_0+\P^k)$. 

(ii)~If $x_0\notin\P^{k-1}$, then $x_0\notin\P^{k}$. Therefore,  for any $x\in x_0+\P^{k-1}$ or $x\in x_0+\P^{k}$, $|x|=|x_0|$. Hence, $w(x_0+\P^k)=|x_0|^\alpha q^{-k}$ and $w(x_0+\P^{k-1})=|x_0|^\alpha q^{-(k-1)}$ so that $w(x_0+\P^{k-1})=q\cdot w(x_0+\P^k)$.

Hence, $|x|^\alpha$ is a doubling weight for $\alpha>-1$.

\subsection{Examples of doubling weights which are not $A_p$ weights}
Here we see that $|x|^\alpha$ is an $A_p$ weight if and only if $-1<\alpha<p-1$. As above, for any ball $x_0+\P^k$ we consider two cases: $x_0\in\P^k$ and $x_0\not\in\P^k$. 

In the first case, by~\eqref{eq}, we have
\begin{eqnarray*}
\lefteqn{
\Bigl(\frac{1}{|\P^k|}\int_{x_0+\P^k}|x|^\alpha\,dx\Bigr)\Bigl(\frac{1}{|\P^k|}\int_{x_0+\P^k}|x|^{-\frac{\alpha}{p-1}}\,dx\Bigr)^{p-1}
} \\
&  & \qquad\qquad\qquad =\frac{(q-1)^p}{q^{\alpha+1}-1}\cdot\Bigl(q^{-\frac{\alpha}{p-1}+1}-1\Bigr)^{1-p}.
\end{eqnarray*}
Note that $-\frac{\alpha}{p-1}+1>0$ since $-1<\alpha<p-1$. In the second case, it is easy to see that 
\[
\Bigl(\frac{1}{|\P^k|}\int_{x_0+\P^k}|x|^\alpha\,dx\Bigr)\Bigl(\frac{1}{|\P^k|}\int_{x_0+\P^k}|x|^{-\frac{\alpha}{p-1}}\,dx\Bigr)^{p-1}=1.
\]
Hence, it follows that for $-1<\alpha<p-1$, $|x|^\alpha$ is an $A_p$ weight. In particular, for $\alpha=(p-1)(1-\theta)$, $0<\theta<1$, we have $[w]_{A_p}\sim\frac{1}{(q^{\theta}-1)^{p-1}}$ since $\frac{1}{q^p}\leq\frac{1}{q^{\alpha+1}-1}\leq 1$.

\begin{remark}
Note that these power weights are similar in structure to such weights in $\R^n$. Our proofs in the above examples are different. Moreover, in contrast to the Euclidean cases, we are getting equality of the form $w(x_0+\P^{k-1})=Cw(x_0+\P^k)$. 
\end{remark}

Now, we show that $[w]_{A_p}^{\frac{1}{p-1}}$ is best possible in the sense of Buckley. Let $w(x)=|x|^{(p-1)(1-\theta)}$, $0<\theta<1$. Then, as we noted above, $[w]_{A_p}\sim\frac{1}{(q^{\theta}-1)^{p-1}}$. Put $f(x)=|x|^{\theta-1}{\mathbf 1}_{\D}$. Let $x_0\in\D$. Then, $|x_0|=q^{-k}$ for some $k\in\N_0$. It is easy to see that $\frac{1}{|\P^k|}\int_{\P^k}|y|^{\theta-1}\,dy\geq (1-q^{-1})\frac{1}{q^{\theta}-1}f(x_0)$. Hence, for all $x\in K$, $Mf(x)\geq (1-q^{-1})\frac{1}{q^{\theta}-1}f(x)$. Therefore,
\[
\|Mf\|_{L^p(w)}\geq(1-q^{-1})\frac{1}{q^{\theta}-1}\|f\| _{L^p(w)}\sim[w]_{A_p}^{\frac{1}{p-1}}\|f\| _{L^p(w)}.
\]
Thus, we obtain the following result on the Hardy-Littlewood maximal operator on local fields.	

\begin{theorem}\label{max}
Let $M$ be the Hardy-Littlewood maximal operator on a local field $K$. If $w\in A_p$, then $M$ is bounded on $L^p(w)$ and $\|M\|_{L^p(w)\rightarrow L^p(w)}\sim[w]_{A_p}^\frac{1}{p-1}$. Moreover, the power $[w]_{A_p}^\frac{1}{p-1}$ is best possible.
\end{theorem}


\section{Proof of Theorem~\ref{main}}\label{s.proof}
In this section we prove our first main result. Recall that the Dirichlet kernels $D_n$, $n\geq 0$, are defined by
\[
D_n(x)=\sum_{k=0}^{n-1}\chi_k(x),\quad n\geq 1,\quad D_0\equiv 0.
\]
Let $\Phi_0$ be the characteristic function of $\D$. For functions on $\D$, we treat them as functions defined on $K$ but supported on $\D$. Then, by this convention, we have $D_n=\Phi_0 D_n$ and $S_{n}f=f*D_n$, where the integration that defines the convolution is over all of $K$. Define
\[
\widetilde{D_n}=\overline{\chi}_nD_n\quad\mbox{and}\quad T_{n}f=\widetilde{D_n}*f. 
\]
The functions $\widetilde{D_n}$, $n\geq 0$, are called the \emph{modified Dirichlet kernels}. If $x\in\D$, then
\begin{eqnarray*}
S_n f(x)
& = & \int_{\D}f(y)D_n(x-y)\,dy \\
& = & \chi_n(x)\int_{\D}f(y)\overline{\chi}_n(y)\overline{\chi}_n(x-y)D_n(x-y)\,dy \\
& = & \chi_n(x)\int_{\D}\bigl(\overline{\chi}_nf\bigr)(y) \widetilde{D_n}(x-y)\,dy.
\end{eqnarray*}
Therefore,
\begin{equation}\label{bridge}
S_{n}f=\chi_{n}\bigl(T_{n}(\overline{\chi}_nf)\bigr). 
\end{equation}
By our convention, $\widetilde{D_n}=\Phi_0\widetilde{D_n}$ is the kernel of the operator $T_n$. We denote this kernel by $K_n$. 

\noindent\emph{Proof of {\rm(a)} implies {\rm(b)}}. Let $1<p<\infty$, $w\in A_p$, and $f\in L^p(\D,w)$. Then $f\in L^1(\D)$ since
\begin{eqnarray*}
\int_{\D}|f(x)|\,dx 
& \leq & \Bigl(\int_{\D}|f(x)|^p w(x)\,dx\Bigl)^\frac{1}{p} \Bigl(\int_{\D} w(x)^{-\frac{1}{p-1}}\,dx\Bigr)^\frac{1}{p'} \\
& \leq & \|f\|_{L^p(\D,w)}[w]_{A_p}^\frac{1}{p}\Bigl(\frac{1}{w(\D)}\Bigr)^\frac{1}{p} <\infty,
\end{eqnarray*}
by H\"older's inequality and the fact that $w\in A_p$. Hence, we can define the Fourier coefficients of $f$ so that $S_{n}f$ makes sense. By~\eqref{bridge}, in order to prove~(b), it is enough to prove that 
\begin{equation}\label{e.weightTn}
\int_{\D}|T_{n}f(x)|^p w(x)\,dx\leq C\int_{\D}|f(x)|^pw(x)\,dx,\quad n=0,1,2,\dots,
\end{equation}
and the constant $C$ is independent of $n$. 

We will first prove~\eqref{e.weightTn} for functions in $C(\D)$, the space of continuous functions on $\D$. Then, we will extend this inequality for a general function in $L^p(\D,w)$ by a limiting argument.

Let $f$ be a function in $C(\D)$. We note that $T_nf$ is supported in $\D$. In~\cite{Taib1}, Taibleson proved that the operators $T_n$, $n\in\N_0$, are uniformly bounded on $ L^p(\D)$. That is, there is a constant $C_p>0$, independent of $n$, such that
\begin{equation}\label{e.taibleson}
\|T_{n}f\|_{L^p(\D)}\leq C_{p}\|f\|_{L^p(\D)}\quad\mbox{for all}~f\in L^p(\D).
\end{equation}
Using this result and the reverse H\"older inequality for $A_p$ weights, it can be seen that $T_{n}f \in L^p(\D,w)$. Indeed, we take the $\epsilon$ obtained from Theorem~\ref{rhi} and let $q'=1+\epsilon$. Then we have
\begin{eqnarray*}
\int_{\D}|T_n f(x)|^p w(x)\,dx
& \leq & \Bigl(\int_{\D}|T_n f(x)|^{pq}\,dx\Bigr)^{\frac{1}{pq}\cdot p}\Bigl(\int_{\D}w(x)^{q'}\,dx\Bigr)^{\frac{1}{q'}} \\
& \leq & C_{pq}\|f\|_{L^{pq}(\D)}^p\cdot C\int_{\D}w(x)\,dx<\infty.
\end{eqnarray*}
Hence $T_{n}f\in L^p(w)$, $1<p<\infty$. Here we have used the fact that $f\in C(\D)$ so that $f\in L^p(\D)$ for all $p$. From Theorem~\ref{max}, it follows that $M(T_{n}f)\in L^p(w)$. Using Lebesgue differentiaton theorem (see Theorem~1.14, Chapter~II in~\cite{Taib}) and applying Theorem~\ref{T.sharp}, we get
\begin{eqnarray}
\int_{\D}|T_n f(x)|^p w(x)\,dx\nonumber 
& \leq & \int_{K}M\bigl(T_n f(x)\bigl)^p w(x)\,dx\nonumber \\
& \leq & C_{p,w}\int_{K}\bigl((T_n f)^\sharp(x)\bigr)^p w(x)\,dx.
\end{eqnarray}

At this step we need the following pointwise relation between $(T_nf)^\sharp$ and $M_s f$. Recall that for $1<s<\infty$, $M_s f(x)=\bigl(M|f|^s(x)\bigr)^\frac{1}{s}$.  

\begin{lemma}\label{l.tnmr}
Let $1<s<\infty$. Then there is a constant $C_s>0$ such that for any $f\in C(\D)$ and $n=1, 2, \dots$, we have
\begin{equation}\label{ptwisel}
(T_nf)^\sharp(x)\leq C_s M_s f(x)\quad\mbox{for a.e.}\,x.
\end{equation}
\end{lemma}

We postpone the proof of this lemma and continue with the proof of (a) implies (b). Once Lemma~\ref{l.tnmr} is proved,~\eqref{e.weightTn} will follow easily. Indeed, we have 
\begin{eqnarray}
\int_{\D}|T_n f(x)|^p w(x)\,dx
& \leq & C_{p,w} \int_{K}\bigl((T_n f)^\sharp(x)\bigr)^p w(x)\,dx \nonumber \\
& \leq & C_{p,w}{C_s}^p\int_{K}{M_s f(x)}^p w(x)\,dx \nonumber \\
& = & C_{p,w} {C_s}^p \int_{K}(M|f|^s(x))^\frac{p}{s}w(x)\,dx.
\end{eqnarray}
Now, from Corollary~\ref{cor2}, we obtain an $r$ with $1<r<p$ such that $w\in A_r$. We choose $s=\frac{p}{r}$. Then $\frac{p}{s}=r>1$ and $w\in A_{\frac{p}{s}}$. Since $|f|^s\in L^{\frac{p}{s}}(w)$, we apply Theorem~\ref{lca} and get
\[
\int_{\D}|T_n f(x)|^p w(x)\,dx 
\leq C_{p,w} {C_s}^p C_p \int_{\D}{|f(x)|}^p w(x)\,dx.
\]
We also note that the constant in the above inequality is independent of $n$. Hence, accepting the validity of  Lemma~\ref{l.tnmr}, we have proved~\eqref{e.weightTn} for all $f\in C(\D)$. 

Let $1<p<\infty$. Since $d\mu(x)=w(x)\,dx$ is a regular measure, $C_c(K)$ is dense in $L^p(w)$ so that $C(\D)$ is dense in $L^p(\D, w)$. Now, we extend $T_n$ from $C(\D)$ to $L^p(\D,w)$ by continuity and call this operator $\tilde T_n$. For $f\in L^p(\D, w)$, choose a sequence $\{f_l\}$ in $C(\D)$ converging to $f$. Hence, $T_nf_l\rightarrow\tilde{T_n}f$ in $L^p(\D, w)$ as $l\rightarrow\infty$. Also, since $T_n$ is the convolution operator with kernel $K_n$, it follows that $T_nf_l$ also converges to $T_nf$ in $L^p(\D,w)$ by Young's inequality. This show that $\tilde{T_n}$ coincides with $T_n$. 

Therefore, to complete the proof of (a) implies (b), it remains to show that Lemma~\ref{l.tnmr} is valid. To prove this lemma we need the following result. Recall that $\mathcal{S}$ is the space of all finite linear combinations of functions of the form $\tau_h\Phi_k$, $h\in K$, $k\in\Z$. 

\begin{proposition}\label{p.constcosets}
Let $\varphi$ be a function in $\mathcal{S}$. If $\varphi$ is constant on cosets of $\P^{k+1}$ in $\P^k\setminus\P^{k+1}$ for all $k\in\Z$, then $\hat\varphi$ is constant on cosets of $\P^{k+1}$ in $\P^k\setminus\P^{k+1}$ for all $k\in\Z$.
\end{proposition}

\begin{proof}
Observe that the property $\varphi$ is constant on cosets of $\P^{k+1}$ in $\P^k\setminus\P^{k+1}$ for all $k\in\Z$ is equivalent to the statement that $\varphi(x+y)=\varphi(x)$ whenever $|x|>|y|$. 

Since $\varphi\in\mathcal{S}$, there exists $N\in\N$ such that supp~$\varphi\subseteq\P^{-N}$. By Proposition~\ref{supp}, $\hat\varphi$ is constant on cosets of $\P^N$. Hence, it is enough to prove the result for $\hat\varphi_l$, where $\varphi_l=\varphi\cdot{\bf 1}_{\P^{-l}\setminus\P^{-l+1}}$ for all $l\in\Z$. Fix $l\in\Z$. Since $\varphi_l$ is supported on $\P^{-l}$ and constant on cosets of $\P^{-l+1}$, by Proposition~\ref{supp}, $\hat\varphi_l$ is supported on $\P^{l-1}$ and constant on cosets of $\P^{l}$. 

If $|x|>q^{-l+1}$, then $\hat\varphi_l(x)=0$. Then for each $y$ with $|y|<|x|$, we have $|x+y|=|x|$ so that $\hat\varphi_l(x+y)=0$. Now, if $|x|\leq q^{-l+1}$ and $|y|<|x|$, then $|y|\leq q^{-l}$. Since $x\in\P^{l-1}$, we have $x\in\P^l+a$ for some $a\in K$. So $\P^l+a=B(x,q^{-l})$, by Proposition~\ref{p.balls}\,(a). Hence, $|x+y-x|=|y|\leq q^{-l}$ so that $x+y\in B(x,q^{-l})=\P^l+a$. This shows that $x+y$ and $x$ are in the same coset of $\P^l$. Hence, $\hat\varphi_l(x+y)=\hat\varphi_l(x)$. 
\end{proof}

We derive some crucial properties of the kernel $K_n$ which will be needed in the proof of Lemma~\ref{l.tnmr}. We provide a brief sketch of the proof of the following proposition and refer to~\cite{MB} for the details. 

\begin{proposition}\label{p.kernelKn}
Let $K_n=\Phi_0\widetilde{D_n}$ be the kernel of the operator $T_n$. Then
\begin{enumerate}
\item[(a)] $|K_{n}(x)|\leq\frac{q}{|x|}$ for all $x\in K^*$,
\item[(b)] $\widehat{K_n}(x+y)=\widehat{K_n}(x)$ if $|y|<|x|$. 
\end{enumerate}
\end{proposition}

\begin{proof}
(a) Since $K_n$ is supported on $\D$, we take $x\in\D$ and $x \neq 0$. Then, $|x|=q^{-l+1}$ for some $l\geq 1$. We write $n=rq^l+t$, $0 \leq t<q^l$. It is easy to verify that 
\[
D_n(x)=D_r(\p^{-l}x)\cdot D_{q^l}(x)+\chi_r(\p^{-l}x)D_t(x). 
\]
Since $|x|=q^{-l+1}$, we have $D_{q^l}(x)=0$. Therefore,
\[
|D_n(x)|=|\chi_r(\p^{-l}x)D_t(x)|\leq t<q^l=\frac{q}{|x|}.
\]
Hence, we get $|K_{n}(x)|\leq\frac{q}{|x|}$ for all $x\in K^*$.

(b) We have
\[
K_n=\Phi_0\overline{\chi}_nD_n =\Phi_0\overline{\chi}_n\sum_{m=0}^{n-1}\chi_m 
=\sum_{m=0}^{n-1}\Phi_0\chi(u(m)-u(n)).
\]
Since $(\Phi_0\chi_y)^\wedge=\tau_y\Phi_0$, we get
\[
\widehat{K_n}=\sum_{m=0}^{n-1}\tau_{u(m)-u(n)}\Phi_0=\sum_{m=0}^{n-1}{\bf 1}_{\D+u(m)-u(n)}.
\]
That is, $\widehat{K_n}$ is the characteristic function of the union of $n$ disjoint cosets 
\begin{equation}\label{e.cosets}
\{\D+u(m)-u(n):m=0,1,\dots,n-1\}
\end{equation}
of $\D$. Note that we can write $q^k\leq n\leq q^{k+1}-1$ for some $k\in\N_0$. In order to prove (b), it is equivalent to show that the union of the cosets in~\eqref{e.cosets} either contains both $x$ and $x+y$ or neither. We will prove this by induction on $k$. Observe that if $|y|<|x|$, then $|x+y|=|x|$, by~\eqref{e.max}.

If $k=0$, then $1\leq n\leq q-1$. Hence, $\{\D+u(m)-u(n):m=0,1,\dots,n-1\}$ consists of $n$ distinct cosets of $\D$ in $\P^{-1}$ since $u(m)\not=u(n)$ and $|u(m)-u(n)|=q$ for $m=0,1,\dots,n-1$. Using this fact, we can show that the induction hypothesis is true for $k=0$.

Now, assume that the assertion holds for $n<q^k$. We will prove it for all $n$ such that $q^k\leq n\leq q^{k+1}-1$. We write
\[
n=rq^k+s,\quad 1\leq r\leq q-1,\quad 0\leq s\leq q^k-1.
\]
Then $u(n)=u(rq^k)+u(s)$, by~\eqref{eq.un}. If $0\leq m\leq rq^k-1$, then
\[
m=lq^k+t,\quad 0\leq l\leq r-1,\quad 0\leq t\leq q^k-1
\]
so that $u(m)=u(lq^k)+u(t)$. Hence,
\[
u(m)-u(n)=\Bigl(u(lq^k)-u(rq^k)\Bigr)+\Bigl(u(t)-u(s)\Bigr).
\]
If $rq^k\leq m\leq rq^k+s-1$, then $m=rq^k+\nu$, $0\leq \nu\leq s-1$ and $u(m)-u(n)=u(rq^k)+u(\nu)-u(rq^k)-u(s)=u(\nu)-u(s)$. Therefore, the union of the cosets in~\eqref{e.cosets} is the union of 
\begin{equation}\label{e.cosets1}
\bigcup_{l=0}^{r-1}\bigcup_{t=0}^{q^k-1}\Bigl(\D+u(lq^k)-u(rq^k)+u(t)-u(s)\Bigr)
\end{equation}
and
\begin{equation}\label{e.cosets2}
\bigcup_{\nu=0}^{s-1}\bigl(\D+u(\nu)-u(s)\bigr).
\end{equation}
Since $s<q^k$, the cosets in~\eqref{e.cosets2} satisfy the induction hypothesis. Hence, both $x$ and $x+y$ belong to this union or neither does. For the cosets in~\eqref{e.cosets1}, we observe that $|u(t)-u(s)|\leq q^k$ for $t=0,1,\dots,q^k-1$ so that $\bigcup_{t=0}^{q^k-1}\Bigl(\D+u(t)-u(s)\Bigr)=\P^{-k}$. Hence, the union in~\eqref{e.cosets1} is
\[
\bigcup_{l=0}^{r-1}\Bigl(\P^{-k}+u(lq^k)-u(rq^k)\Bigr).
\]
This is a union of $r$ cosets of $\P^{-k}$ in $\P^{-k-1}\setminus\P^{-k}$. If $x$ is in any of these $r$ cosets, then $|x|=q^{k+1}$ and if $|y|<|x|$, then $|y|\leq q^k$. Then, it follows that $x$ and $x+y$ are in the same coset of $\P^{-k}$ and the induction is complete.  
\end{proof}

We are now ready to prove Lemma~\ref{l.tnmr}.

\begin{proof}[Proof of Lemma~\ref{l.tnmr}]
First we observe that  
\begin{equation}\label{e.fstar}
\|f^\sharp\|_\infty\leq 2\sup_B\inf_{\alpha\in\C}\frac{1}{|B|}\int_B|f(x)-\alpha|\,dx,
\end{equation}
where the supremum is taken over all balls $B$. To see this, we note that for all $\alpha\in\C$,
\[
\int_{B}|f(x)-f_B|\,dx\leq\int_{B}|f(x)- \alpha|\,dx+\int_{B}|\alpha-f_B|\,dx\leq 2\int_{B}|f(x)-\alpha|\,dx.
\]
Now, we divide the both sides by $|B|$, take the infimum over $\alpha\in\C$ and supremum over all balls $B$ to get~\eqref{e.fstar}.

Fix $s>1$, $x\in K$ and any ball $B$ containing $x$. By~\eqref{e.fstar}, it is enough to show that there exists a constant $\alpha$ such that
\begin{equation}\label{e.TnMs}
\frac{1}{|B|}\int_{B}|T_n f(y)- \alpha|\,dy\leq C M_s f(x).
\end{equation}
For $l\in\Z$, consider the integral 
\[
I_l(x)=\int_{|y-x|\geq q^l} K_n(x-y)f(y)\,dy.
\]
This integral exists since $|K_{n}(x)|\leq\frac{q}{|x|}$ for all $x\in K^*$, by Proposition~\ref{p.kernelKn}. 

Now, we decompose $f$ as follows: $f=f_1+f_2$, where $f_1=f\cdot{\bf 1}_{B(x,q^{l-1})}$. Then, by linearity of $T_n$, it follows that
\begin{equation}\label{e.TnIl}
|T_n f(y)-I_l(x)|\leq |T_n f_1(y)|+|T_n f_2(y)-I_l(x)|.
\end{equation}  
Since $s>1$, $\{T_n:n\in\N_0\}$ is uniformly bounded on $L^s(\D)$ (see~\eqref{e.taibleson}). By H\"older's inequality, the average of $|T_n f_1|$ over the ball $B(x, q^{l-1})$ is
\begin{eqnarray}
\frac{1}{q^{l-1}}\int_{|y-x|\leq q^{l-1}}|T_n f_1(y)|\,dy
& \leq & \frac{1}{q^{l-1}}\Bigl( \int_{|y-x|\leq q^{l-1}}|T_n f_1(y)|^s\,dy\Bigr)^\frac{1}{s}\cdot(q^{l-1})^\frac{1}{s'}\nonumber\\
& \leq & \frac{C}{q^{l-1}}\Bigl(\int_{|y-x|\leq q^{l-1}}|f(y)|^s\,dy\Bigr)^\frac{1}{s}\cdot(q^{l-1})^\frac{1}{s'}\nonumber\\
& \leq & C M_sf(x).\label{e.Tnf1}
\end{eqnarray}
The average of the second function on the right in~\eqref{e.TnIl} over the ball $B(x, q^{l-1})$ is dominated by
\begin{equation}\label{e.kernel2}
\int_{|y-x|\leq q^{l-1}}\int_{|z-x|\geq q^l}\Bigl|K_n(y-z)-K_n(x-z)\Bigr||f(z)|\,dz\,dy.
\end{equation}
We will now show that this integral is zero. 

Since $\widehat{K_n}=\sum_{m=0}^{n-1}\tau_{u(m)-u(n)}\Phi_0$, by definition $\widehat{K_n}\in{\mathcal S}$. 
By Proposition~\ref{p.kernelKn}\,(b), we have $\widehat{K_n}(x+y)=\widehat{K_n}(x)$ if $|y|<|x|$. Now, applying Proposition~\ref{p.constcosets} to $\widehat{K_n}$ and observing that $\widehat{\widehat{K_n}}$ is the reflection of $K_n$, we conclude that 
\[
K_n(x+y)=K_n(x)\quad\mbox{whenever}~|y|<|x|.
\]
In~\eqref{e.kernel2}, $|y-x|<|x-z|$ so that $K_n(y-z)=K_n(x-z)$ and hence the integral in~\eqref{e.kernel2} is zero. From this fact and the estimate in~\eqref{e.Tnf1}, it follows that
\[
\frac{1}{q^{l-1}}\int_{|y-x|<q^l}|T_n f(y)-I_l(x)|\,dy\leq C M_s f(x).
\]
Hence,~\eqref{e.TnMs} is satisfied with $\alpha=I_l(x)$. This completes the proof of the lemma.
\end{proof}
	
\begin{proof}[Proof of {\rm(b)} implies {\rm(a)}]
First we observe that, it is enough to show that inequality~\eqref{nec} holds for all balls $B$ with $|B|\leq 1$. In fact, let $B$ be any ball with $|B|>1$, then $|B|=q^k$ for some $k\geq 1$. Hence, $B$ can be written as a disjoint union of $q^k$ cosets of $\D$ as $B=\bigcup_{i=1}^{q^k}(u(l_i)+\D)$, where $l_i\in\N_0$. We observe that
\[
\frac{1}{|B|}\int_{B}w(x)\,dx=\frac{1}{|B|}\sum_{i=1}^{q^k}\int_{\D} w(x)\,dx=\int_{\D}w(x)\,dx,
\] 
since $w$ is $\Lambda$-periodic. This reduces to the case when $|B|=1$. Therefore, we assume that $|B|\leq 1$. Then, $|B|=q^{-r}$ for some $r\in\N_0$ and $B\subset u(l)+\D$ for some $l\in\N_0$. Let $f$ be a non-negative function on $B$ and $0$ on $\bigr(u(l)+\D\bigl)\setminus B$. Extend $f$ to $K$ $\Lambda$-periodically. Now, for any $x\in B$, we have $B=x+\P^r$. If $y\in B$, then $y-x\in\P^r$. Using the fact that $D_{q^r}=q^r{\bf 1}_{\P^r}$, $r\geq 0$ (see Lemma~6.7, Chapter~II in~\cite{Taib}), we get
\[
S_{q^r}f(x)=\int_B f(y)D_{q^r}(x-y)\,dy=\int_B f(y)q^r\,dy=\frac{1}{|B|} \int_{B}f(y)\,dy.
\] 
Hence, by~\eqref{weightS_nf}, we get
\[
\Bigl(\int_{B}w(x)\,dx\Bigr)\Bigl(\frac{1}{|B|}\int_{B}f(y)\,dy\Bigr)^p \leq C\int_B|f(x)|^p w(x)\,dx.
\] 
Then, by a standard argument as in page 247 in~\cite{HMW}, we get~\eqref{nec} from the above inequality.
\end{proof}

\begin{proof}[Proof of {\rm (b)} implies {\rm (c)}]
First, note that $w\in L^1(\D)$ as we have already proved that (b) implies (a). We have also  observed in the beginning of section~\ref{s.proof} that if $f\in L^p(\D, w)$, then $f\in L^1(\D)$. Let $V=\{\sum_k c_k\chi_k:k\in\N_0\}$, the vector space of all finite linear combination of characters of $\D$. Then, for any $g \in V$, there exists $n\in\N_0$ large enough, depending on $g$, such that $S_n g=g$. Therefore, it suffices to show that the space $V$ is dense in $L^p(\D,w)$. 

In the proof (a) implies (b), we have shown that $C(\D)$ is dense in $L^p(\D,w)$. Next, we show that $V$ is dense in $C(\D)$ in the $L^\infty$-norm. In fact, for any $f\in C(\D)$, $l\geq 0$ and $x\in\D$, we have, 
\[
|S_{q^l}f(x)-f(x)|=\Bigl|q^l\int_{|z-x|\leq q^{-l}}\bigl(f(x-z)-f(x)\bigr)\,dz\Bigr|.
\]
Now, from the uniform continuity of $f\in C(\D)$, for any $\epsilon>0$, there exists $k>0$ such that $|f(x-z)-f(x)|<\epsilon$ whenever $|z| \leq q^{-k}$. Using this in the above equation, we see that $|S_{q^l} f(x)-f(x)|<\epsilon$ for $l\geq k$, independent of $x$. Actually, here the convergence is uniform. Since $w\in L^1(\D)$, it is easy to see that the space $V$ is dense in $L^p(\D,w)$ and hence (c) follows.
\end{proof}
 
\begin{proof}[Proof of {\rm (c)} implies {\rm (b)}]
Suppose (c) is true. An argument similar to the one in page~246 of~\cite{HMW} shows that both $w$ and $w^{-\frac{1}{p-1}}$ are in $L^1(\D)$, and hence can be omitted. We need to show that the maps $S_n:L^p(\D,w)\rightarrow L^p(\D,w)$ are uniformly bounded. By uniform boundedness principle, it is enough to show that $\sup_n\|S_n f\|_{L^p(\D,w)}<\infty$ for every $f\in L^p(\D,w)$. Proof of this fact can also be adapted from that in~\cite{HMW} with suitable modifications.  
\end{proof}


\section{Application to Schauder bases}\label{s.appl}
Let $K$ be a local field of positive characteristic and $\varphi\in L^2(K)$. Define
\[
V_\varphi=\overline{\rm span}\{\varphi(\cdot-u(k)): k\in\N_0\},
\]
the closure in $L^2(K)$ of the finite linear combinations of translates of $\varphi$ by elements of $\Lambda=\{u(k): k\in\N_0\}$. Such a space is called a \emph{principal shift-invariant space}. In general, a closed subspace $V$ of $L^2(K)$ is called a \emph{shift-invariant space} if $f(\cdot-u(k))\in V$ for all $f\in V$ and $k\in\N_0$. Shift-invariant spaces play a very important role in the study of wavelets. For example, in a multiresolution analysis, the space $V_0$ is a principal shift-invariant space and the wavelet space $W_0$ is, in general, a shift-invariant space. In~\cite{BB}, we used the properties of shift-invariant spaces to provide a characterization of wavelets in a local field of positive characteristic.

In this section, we characterize all $\varphi\in L^2(K)$ for which the system of translates $\{\varphi(\cdot-u(k)): k\in\N_0\}$ forms a Schauder basis for $V_\varphi$. Here, $\N_0$ is ordered with the usual order, i.e., as $0, 1, 2,\dots$. The reason for taking $K$ to be of positive characteristic is that the translation set $\Lambda=\{u(k): k\in\N_0\}$ forms a subgroup of the additive group $K^+$ (see Proposition~\ref{p.un}) so that the periodization of $|\hat{\varphi}|^2$, given by
\[
w_\varphi(\xi)=\sum_{k\in\N_0}|\hat{\varphi}(\xi+u(k))|^2, 
\]
is $\Lambda$-periodic. Indeed, by Proposition~\ref{p.un}(c), for $l\in\N_0$, we have
\[
w_\varphi(\xi+u(l))=\sum_{k\in\N_0}|\hat{\varphi}(\xi+u(l)+u(k))|^2=\sum_{k\in\N_0}|\hat{\varphi}(\xi+u(k))|^2=w_\varphi(\xi)
\]
for a.e.\,$\xi\in K$.

It can be shown that the map $J_\varphi:L^2(\D,w_\varphi)\rightarrow V_\varphi$, given by $J_\varphi f=(f\hat\varphi)^\vee$, is an isometry, where $f^\vee$ is the inverse Fourier transform of $f$. For a proof of this fact, we refer to~\cite{BB}. 

Note that $(J_\varphi\overline{\chi}_k)^\wedge=\overline{\chi}_k\hat\varphi=[\varphi(\cdot-u(k))]^\wedge$ so that $J_\varphi$ maps $\overline{\chi}_k$ to $\varphi(\cdot-u(k))$. Thus, various properties of $\{\varphi(\cdot-u(k)): k\in\N_0\}$ on $V_\varphi$ correspond to similar properties of the system $\{\chi_k:k\in\N_0\}$ on $L^2(\D,w_\varphi)$. Therefore, our original problem is equivalent to the problem of finding conditions on $w_\varphi$ so that $\{\chi_k:k\in\N_0\}$ forms a Schauder basis for $L^2(\D,w_\varphi)$. 

Let us recall some standard facts about bases in a Banach space $\mathbb{B}$. A sequence $\{x_k: k\in\N_0\}$ of elements of $\mathbb{B}$ is called a \emph{Schauder basis} for $\mathbb{B}$ if for every $x\in\mathbb{B}$ there exists a unique sequence $\{\alpha_k: k\in\N_0\}$ of scalars such that
\[
x=\sum_{k\in\N_0}\alpha_kx_k,
\]
where the partial sums of the series converge in the norm of $\mathbb{B}$, that is,
\[
\lim_{N\rightarrow\infty}\Bigl\|x-\sum_{k=0}^N\alpha_kx_k\Bigr\|=0.
\]

Let $\{x_n:n\in\N_0\}$ be a sequence in a Hilbert space $\mathbb{H}$. A sequence $\{\tilde x_n:n\in\N_0\}$ in $\mathbb{H}$ is said to be \emph{biorthogonal} to $\{x_n:n\in\N_0\}$ if $\langle x_k, \tilde x_l\rangle=\delta_{k,l}$ for all $k,l\in\N_0$. It is easy to verify that if $\{x_n:n\in\N_0\}$ is complete in $\mathbb{H}$, that is, if $\overline{\rm span}\{x_n:n\in\N_0\}=\mathbb{H}$, then there is a unique sequence $\{\tilde x_n:n\in\N_0\}\subset\mathbb{H}$ which is biorthogonal to $\{x_n:n\in\N_0\}$. Such a sequence is called the \emph{biorthogonal dual} of $\{x_n:n\in\N_0\}$. Every Schauder basis has a unique biorthogonal dual. 

Let $\mathbb{H}=V_\varphi$. Suppose there exists $\tilde\varphi\in V_\varphi$ such that $\langle\varphi(\cdot-u(k)), \tilde\varphi\rangle=\delta_{k,0}$ for all $k\in\N_0$. Now,
\[
\langle\varphi(\cdot-u(k)), \tilde\varphi(\cdot-u(l))\rangle
=\langle\varphi(\cdot-(u(k)-u(l))), \tilde\varphi\rangle
\]
If $k=l$, then the above inner product is equal to $\langle\varphi, \tilde\varphi\rangle=1$. If $k\not=l$, then $0\not=u(k)-u(l)=u(m)$ for some $m\in\N_0$, by Proposition~\ref{p.un}. Hence, $m\not=0$. So, the inner product is equal to $\delta_{m,0}=0$. Thus, $\langle\varphi(\cdot-u(k)), \tilde\varphi(\cdot-u(l))\rangle=\delta_{k,l}$. That is, if there exists $\tilde\varphi\in V_\varphi$ such that $\langle\varphi(\cdot-u(k)), \tilde\varphi\rangle=\delta_{k,0}$ for all $k\in\N_0$, then $\{\tilde\varphi(\cdot-u(k)):k\in\N_0\}$ is a biorthogonal dual of $\{\varphi(\cdot-u(k)):k\in\N_0\}$. The function $\tilde\varphi$ will then be called the \emph{canonical dual function} to $\varphi$. As above, if it exists, then it is unique. We will need the following result. For a proof, we refer to~\cite{Sin} (see Theorem~4.1, Chapter~1).

\begin{lemma}\label{l.schauder}
A complete sequence $\{x_n:n\in\N_0\}$ with biorthogonal dual $\{\tilde x_n:n\in\N_0\}$ is a Schauder basis for $\mathbb{H}$ if and only if the partial sum operators
\[
s_n(x)=\sum_{k=0}^{n-1}\langle x,y_k\rangle x_k
\]
are uniformly bounded in $\mathbb{H}$.
\end{lemma}

The following result provides a necessary and sufficient condition for the existence of a canonical dual.

\begin{proposition}\label{p.cdual}
Let $\varphi\in L^2(K)$. There exists a canonical dual $\tilde\varphi$ of $\varphi$ in $V_\varphi$ if and only if $\frac{1}{w_\varphi}\in L^1(\D)$. In this case, $\tilde\varphi=(\frac{1}{w_\varphi}\hat\varphi)^\vee$.
\end{proposition}

\begin{proof}
Since the map $J_\varphi:L^2(\D,w_\varphi)\rightarrow V_\varphi$ is an isometry, $\tilde\varphi\in V_\varphi$ if and only if there exists a unique $m$ in $L^2(\D,w_\varphi)$ such that $\widehat{\tilde\varphi}=m\hat\varphi$. Moreover, to be a canonical dual, $\tilde\varphi$ must satisfy $\langle\varphi(\cdot-u(k)), \tilde\varphi\rangle=\delta_{k,0}$ for all $k\in\N_0$. But
\begin{eqnarray*}
\langle\varphi(\cdot-u(k)), \tilde\varphi\rangle
& = & \int_K\hat\varphi(\xi)\overline{\chi_k(\xi)}\overline{\widehat{\tilde\varphi}(\xi)}\,d\xi \\
& = & \int_K\overline{m(\xi)}|\hat\varphi(\xi)|^2 \overline{\chi_k(\xi)}\,d\xi \\
& = & \int_{\D}\overline{m(\xi)}w_\varphi(\xi)\overline{\chi_k(\xi)}\,d\xi \\
& = & (\overline{m}w_\varphi)^\wedge(u(k)).  
\end{eqnarray*}
Thus, the $k$th Fourier coefficient of $\overline{m}w_{\varphi}$  is equal to $\delta_{k,0}$ for all $k\in\N_0$. This will happen
if and only if $\overline{m}w_{\varphi}=1$ for a.e.\,$\xi\in\D$. Since $w_{\varphi}$ is real-valued, we have $\overline{m}=m=\frac{1}{w{_\varphi}}$. Finally, $m=\frac{1}{w_\varphi}\in L^2(\D,w_\varphi)$ if and only if $\int_{\D}\frac{1}{w_\varphi^2(\xi)}w_\varphi(\xi)       \,d\xi=\int_{\D}\frac{1}{w_\varphi(\xi)}\,d\xi<\infty$ if and only
if $\frac{1}{w_\varphi}\in L^1(\D)$.
\end{proof}

A weight $w$ is said to be an \emph{$A_2(\D)$ weight} if $w$ is $\Lambda$-periodic and there exists a constant $C>0$ such that
\[
\Bigl(\frac{1}{|B|}\int_{B}w(x)\,dx\Bigr)\Bigl(\frac{1}{|B|}\int_{S}w(x)^{-1}\,dx\Bigr)\leq C
\]
for all balls $B\subseteq\D$. In this case, we say that $w\in A_2(\D)$. 

We are now ready to prove Theorem $\ref{sc}$.

\begin{proof} [Proof of Theorem \ref{sc}]
Let $\{\varphi(\cdot-u(k)):k\in\N_0\}$ be a Schauder basis for $V_\varphi$. Then, since $J_\varphi$ is an isometry, $\{\chi_k: k\in\N_0\}$ is a Schauder basis for $L^2(\D, w_\varphi)$. Let $\{z_k: k\in\N_0\}$ be the biorthogonal dual of $\{\chi_k: k\in\N_0\}$ in $L^2(\D, w_\varphi)$. By Proposition~\ref{p.cdual}, $\frac{1}{w_\varphi}\in L^1(\D)$. In particular, $w_\varphi>0$ a.e. We have,
\[
\delta_{k,l}
=\langle \chi_k,z_l\rangle_{L^2(\D, w_\varphi)}
=\int_{\D}\chi_k(\xi)\overline{z_l}(\xi) w_\varphi(\xi)\,d\xi= \int_{\D}z_l(\xi)w_\varphi(\xi)\overline{\chi_k}(\xi)\,d\xi. 
\]
Hence, the function $z_lw_\varphi$ has all but the $l$th Fourier coefficient are zero. By the uniqueness of Fourier coefficients, we have $z_lw_\varphi=\chi_l$, $l\in\N_0$. Now, an easy computation shows that for any $f\in L^2(\D, w_\varphi)$, we have 
\begin{equation}\label{e.pso}
\langle f, z_k\rangle_{L^2(\D, w_\varphi)}=\langle f,\chi_k\rangle_{L^2(\D)}.
\end{equation}
Now, we define 
\[
s_n f=\sum_{k=0}^{n-1}\langle f,z_k\rangle_{L^2(\D, w_\varphi)} \chi_k.
\]
Since $\{\chi_k: k\in\N_0\}$ is a Schauder basis for $L^2(\D, w_\varphi)$ with biorthogonal dual $\{z_k: k\in\N_0\}$, by Lemma~\ref{l.schauder}, the partial sum operators $s_n:L^2(\D, w_\varphi)\rightarrow L^2(\D, w_\varphi)$, $n\in\N_0$, are uniformly bounded. But, from~\eqref{e.pso}, we see that $s_nf=S_nf$, the usual partial sums of the Fourier series of $f$. Hence, $S_n:L^2(\D, w_\varphi)\rightarrow L^2(\D, w_\varphi)$ are uniformly bounded. Therefore, by Theorem~\ref{main}, it follows that $w_\varphi\in A_2(\D)$.

Conversely, suppose that $w_\varphi\in A_2(\D)$. Then $w_\varphi>0$ a.e.\,and $\frac{1}{w_\varphi}\in L^1(\D)$. Hence, $\{z_k=\frac{\chi_k}{w_\varphi}: k\in\N_0\}$ is the biorthogonal dual of the complete system $\{\chi_k: k\in\N_0\}$. Again, since $\langle f, z_k\rangle_{L^2(\D, w_\varphi)}=\langle f,\chi_k\rangle_{L^2(\D)}$, we see that the operator $s_n$ coincides with the Fourier partial sum operator $S_n$. By Theorem~\ref{main}, $S_n:L^2(\D, w_\varphi)\rightarrow L^2(\D, w_\varphi)$ are uniformly bounded. Hence, $s_n:L^2(\D, w_\varphi)\rightarrow L^2(\D, w_\varphi)$ are also uniformly bounded. Again, by Lemma~\ref{l.schauder}, $\{\chi_k: k\in\N_0\}$ is a Schauder basis for $L^2(\D, w_\varphi)$ which, in turn, shows that $\{\varphi(\cdot-u(k)): k\in\N_0\}$ is a Schauder basis for $V_\varphi$.
\end{proof}

We conclude the article with a proof of Theorem~\ref{onb}.

\begin{proof}[Proof of Theorem~\ref{onb}]
Suppose $\Omega$ tiles $\Q_p$ by translations. Then, by Theorem~\ref{fuglede}, $\Omega$ is a spectral set. That is, there is a subset $\Gamma$ of $\Q_p$ such that $\{\chi_{\gamma}: \gamma \in\Gamma\}$ forms an orthonormal basis for $L^2(\Omega)$. We have to show that the system of translates $\{\varphi(\cdot-\gamma):\gamma\in\Gamma\}$ forms an orthonormal basis for $V(\varphi,\Gamma)$. 

Let $f\in V(\varphi,\Gamma)$. Observe that the Fourier transform $\hat{f}$ of $f$ is supported on $\Omega$ and hence $\hat{f}\in L^2(\Omega)$. Since $|\hat\varphi|=1$ on $\Omega$, $\hat f/\hat\varphi\in L^2(\Omega)$. Now, since $\{\chi_{\gamma}: \gamma\in\Gamma\}$ is an orthonormal basis for $L^2(\Omega)$, we have $\tfrac{\hat f}{\hat\varphi}=\sum_{\gamma\in\Gamma}\langle\tfrac{\hat f}{\hat\varphi},\overline{\chi}_\gamma\rangle\overline{\chi}_\gamma$. Using $|\hat\varphi|^2=1$ and $[\varphi(\cdot-\gamma)]^\wedge=\hat\varphi\overline{\chi}_\gamma$, we get
\[
\tfrac{\hat f}{\hat\varphi}
=\sum_{\gamma\in\Gamma}\bigl\langle\hat f,\tfrac{1}{\overline{\hat\varphi}}\overline{\chi}_\gamma\bigr\rangle\overline{\chi}_\gamma 
=\sum_{\gamma\in\Gamma}\langle\hat f,\hat\varphi\overline{\chi}_\gamma\rangle\overline{\chi}_\gamma 
=\sum_{\gamma\in\Gamma}\langle f,\varphi(\cdot-\gamma)\rangle\overline{\chi}_\gamma.
\]
Hence,
\[
\hat f=\sum_{\gamma\in\Gamma}\langle f,\varphi(\cdot-\gamma)\rangle\hat\varphi\overline{\chi}_\gamma=\sum_{\gamma\in\Gamma}\langle f,\varphi(\cdot-\gamma)\rangle[\varphi(\cdot-\gamma)]^\wedge.
\]
Therefore,
\[
f=\sum_{\gamma\in\Gamma}\langle f,\varphi(\cdot-\gamma)\rangle\varphi(\cdot-\gamma).
\]
To see the orthonormality of $\{\varphi(\cdot-\gamma):\gamma\in\Gamma\}$, we observe that, for $\gamma,\lambda\in\Gamma$, we have
\[
\langle\varphi(\cdot-\gamma),\varphi(\cdot-\lambda)\rangle=\langle\hat\varphi\overline{\chi}_\gamma,\hat\varphi\overline{\chi}_\lambda\rangle=\langle\overline{\chi}_\gamma,\overline{\chi}_\lambda\rangle=\delta_{\gamma,\lambda}.
\]
Here, we have used Parseval's identity and the fact that $|\hat\varphi|^2=1$ on $\Omega$.

Conversely, suppose there is a countable set $\Gamma$ for which $\{\varphi(\cdot-\gamma):\gamma\in\Gamma\}$ forms an orthonormal basis for $V(\varphi,\Gamma)$. Then, by a similar argument as above, $\{\chi_{\gamma}:\gamma\in\Gamma\}$ forms an orthonormal basis for $L^2(\Omega)$ and then by Theorem~\ref{fuglede}, we conclude that $\Omega$ tiles $\Q_p$ by translations.
\end{proof}


\end{document}